\documentclass[a4paper, oneside,11 pt]{article}

\usepackage[latin1]{inputenc}
\usepackage[T1]{fontenc}
\usepackage[english]{babel}
\usepackage{amsfonts}
\usepackage{amssymb}
\usepackage{amsmath,mathrsfs}
\usepackage{lmodern}
\usepackage{amsthm}
\usepackage{graphicx}
\usepackage{vmargin}

\usepackage[all]{xy}
\usepackage{verbatim}
\usepackage{amsmath}
\usepackage{amsthm}
\usepackage{graphicx}
\usepackage{amssymb}
\usepackage{epstopdf}
\usepackage{url}
\usepackage{amsfonts}
\usepackage{todonotes}

\def\bN {\mathbf{N}}

\newcommand{\Div}{\operatorname{div}}

\newcommand{\Id}{\operatorname{Id}}

\newcommand{\ba}{\begin{aligned}}
\newcommand{\ea}{\end{aligned}}

\newcommand{\be}{\begin{equation}}
\newcommand{\ee}{\end{equation}}

\def\eps {{\epsilon}}

\def\grad {{\nabla}}

\newcommand{\pt}{\partial}

\newcommand{\nn}{\nonumber}

\newcommand{\br}{\mathbb{R}}

\renewcommand{\eps}{\varepsilon}

\renewcommand{\(}{\left(}
\renewcommand{\)}{\right)}
\renewcommand{\[}{\left[}
\renewcommand{\]}{\right]}

\newtheorem{thm}{Theorem}
\newtheorem{lem}[thm]{Lemma}
\newtheorem{cor}[thm]{Corollary}
\newtheorem{prop}[thm]{Proposition}
\newtheorem{defi}[thm]{Definition}
\newtheorem{remark}[thm]{Remark}

\def\be{\begin{equation}}
\def\ee{\end{equation}}
\def\bea{\begin{eqnarray}}
\def\eea{\end{eqnarray}}

\numberwithin{thm}{section}
\numberwithin{equation}{section}
\numberwithin{figure}{section}

\begin{document}

\title{On the local uniqueness of steady states\\ for the Vlasov-Poisson system}

\author{
Mikaela Iacobelli
  \thanks{ETH Z\"urich, Mathematics Department, R\"amistrasse 101, 8092 Z\"urich, Switzerland. Email: \textsf{mikaela.iacobelli@math.ethz.ch}}
}

\maketitle

\begin{abstract}
Motivated by the results of
Lemou, M\'ehats, and R\"aphael \cite{LMR} and Lemou \cite{Lem}
concerning the quantitative stability of some suitable
steady states for the Vlasov-Poisson system,
we investigate the local uniqueness of steady states near these ones. This is inspired by analogous results of  Choffrut and  \v{S}ver\'ak in the context of the 2D Euler equations \cite{CS}.
\end{abstract}
\section{Introduction}

The gravitational Vlasov-Poisson equation modelizes the evolution of many particles subject to their own gravity, assuming that both the relativistic effects and the collisions between particles can be neglected.
We consider the Vlasov-Poisson system in three dimensions:
\begin{equation}
\label{gvp}
 \left\{ \begin{array}{ccc}\pt_t f+v\cdot \grad_x f-\grad\phi_f\cdot \grad_v f=0, & (t,x,v)\in\br^+\times\br^3\times\br^3 \\
f(0,x,v)=f_0(x,v)\ge0,\ \  \int f_0\,dx\,dv=1.
\end{array} \right.
\end{equation}
where the Newtonian potential $\phi_f$ is given in terms of the density $\rho_f:$
$$
\rho_f(x)=\int_{\br^3}f(x,v)\,dv, \quad\mbox{and}\quad\phi_f(x)=-\frac{1}{4\pi|x|}*\rho_f=K*\rho_f.
$$

At the beginning of the last century, astrophysicist Sir J. Jeans used this system to model stellar clusters and galaxies \cite{Jea} to study their stability properties. In this context, it appears in many textbooks on astrophysics, such as \cite{Bin,Fri}. In the repulsive case, this system was introduced by A. A. Vlasov around 1937 \cite{Vla1,Vla2}.
Because of the considerable importance in plasma physics and astrophysics,
there is a vast literature on the Vlasov-Poisson system.

The global existence and uniqueness of classical solutions of the Cauchy problem for the Vlasov-Poisson system was obtained by Iordanskii \cite{iord} in dimension 1,  Ukai-Okabe \cite{uo} in the 2-dimensional case, and independently by Lions-Perthame \cite{lp} and Pfaffelmoser \cite{Pfa} in the 3-dimensional case (see also \cite{Sch}). There are currently no results about the existence and uniqueness of classical solutions in dimensions greater than 3.

It is important to mention that, parallel to the existence of classical solutions, there have been a considerable amount of work on the existence
of weak solutions, particularly under very low assumptions on the initial data.

We mention in particular the classical result by Arsen'ev \cite{Ar}, who proved the global existence of weak solutions
under the hypothesis that $f_0$ is bounded  and has finite kinetic energy,
and the result of Horst and Hunze \cite{hh}, where the authors relax the integrability assumption on $f_0$.
If one wishes to relax even more the integrability assumptions on the initial data, then one enters into the framework 
of the so-called renormalized solutions introduced by Di Perna and Lions \cite{dpl,dpl3,dpl2}.
The interested reader is referred to the papers \cite{acf,bbc} and the survey \cite{GPI} for more details and references.

One of the main features of the nonlinear transport flow \eqref{gvp} is the conservation of the total energy
\begin{equation}
\label{eq:H}
\mathcal{H}(f(t))=\frac{1}{2}\int_{\br^6}|v|^2f(t,x,v)\,\,dx\,dv-\frac{1}{2}\int_{\br^3}|\grad \phi_f(t,x)|^2\,dx=\mathcal{H}(f(0))
\end{equation}
and of the Casimir functions: for all $G\in C^1([0,\infty], \br^+)$ such that $G(0)=0$,
$$
\int_{\br^6}G(f(t,x,v))\,\,dx\,dv=\int_{\br^6}G(f_0(x,v))\,\,dx\,dv.
$$

\subsection{Main result}
This work aims to prove a local uniqueness result for steady states of \eqref{gvp}.
In the paper \cite{Lem} (see also \cite{LMR,M}), the author proves quantitative stability inequalities for the gravitational Vlasov-Poisson system that will be crucial in the following.
More precisely, the author considers a class of steady states $\bar f$ to the Vlasov Poisson system, which are decreasing functions of their microscopic energy, and obtains an explicit control of the $L^1$ distance between $\bar f$ and any function $f$ in terms of the energy $\mathcal H(f)-\mathcal H(\bar f)$ and the $L^1$ distance between the rearrangements $\bar f^*$ and $f^*$ of $\bar f$ and $f$, respectively.

In the following, we give some definitions, and we state the local functional inequality in \cite[Theorem 2]{Lem}.
We first recall the notion of equimeasurability and rearrangement.
\begin{defi}
Given two integrable nonnegative functions $f,g:\br^n\to \br$, we say that $f$ and $g$ are equimeasurable
if 
$$
|\{f>s\}|=|\{g>s\}|\qquad \text{for a.e. $s>0$}.
$$
Then, the (radially decreasing) rearrangement $f^*$ of $f$ is defined as the unique radially decreasing function equimeasurable to $f$.
In other words, the level sets of $f^*$ are given by
$$
\{f^*>s\}=B_{r(s)},\qquad \text{with $r(s)>0$ s.t. }\,|B_{r(s)}|=|\{f>s\}|\text{ for a.e. $s>0$.}
$$
\end{defi}

The following important result is proved in \cite[Theorem 2(ii)]{Lem}.
\begin{thm}
\label{thm:lem}
Consider $\bar f$ a compactly supported steady state solution of \eqref{gvp} of the form
\begin{equation}
\label{eq:bar f lemou}
\bar f(x,v)=F(e(x,v)), \quad\text{with } e(x,v)=\frac{|v|^2}{2}+\phi_{\bar f}(x),
\end{equation}
where $F$ is a continuous function from $\br$ to $\br^+$ that satisfies the following monotonicity property: there exists $e_0<0$ such that $F(e)=0$ for $e\ge e_0$ and $F$ is a $C^1$ function on $(-\infty, e_0)$ with $F'<0$ on $(-\infty, e_0).$
Assume that $f\in L^1\cap L^\infty(\br^6)$ has finite kinetic energy and is sufficiently close to a translation of $\bar f$ in the following sense:
\be
\label{eq:close}
\inf_{z\in\br^3}\|\phi_f-\phi_{\bar f}(\cdot -z)\|_{L^\infty}+\|\nabla \phi_f-\nabla \phi_{\bar f}(\cdot -z)\|_{L^2}< R_0,
\ee
for some suitable constant $R_0>0$. 
Then there exists a constant $K_0>0$, depending only on $\bar f$, such that
$$
\inf_{x_0\in \br^3}\|f-\bar f(\cdot-(x_0,0))\|_{L^1}\le \|f^*-\bar f^*\|_{L^1}+K_0 \bigl[\mathcal{H}(f)-\mathcal{H}(\bar f)
+\|f^*-\bar f^*\|_{L^1}\bigr]^{1/2}.
$$
where we denote $\bar f(\cdot-(x_0,0))=\bar f(x-x_0, v)$.
\end{thm}
An immediate consequence is the following estimate, which will be the starting point of our investigation.
\begin{cor}
Let $f,\bar f$ be as in Theorem \ref{thm:lem}.
Assume in addition that $f$ is equimeasurable to $\bar f.$ Then
\begin{equation}\label{eq:local control}
\inf_{x_0\in\br^3}\|f-\bar f(\cdot-(x_0,0))\|_{L^1}^2\le K_0^2 [\mathcal{H}(f)-\mathcal{H}(\bar f)],
\end{equation}
\end{cor}

Let $\bar f$ be the stationary solution as above. Our goal is to understand if, nearby $\bar f$, other stationary solutions of \eqref{gvp} exist. Because stationary solutions of \eqref{gvp} correspond to critical points of $\mathcal H$ with respect to variations of $\bar f$ generated by Hamiltonian flows (see Lemma \ref{lem:stat} below), it makes sense to consider 
a ``neighborhood'' of $\bar f$ generated by flows of smooth Hamiltonians.
Noticing that $\bar f$ is supported in a ball $B_\rho\subset \br^3\times \br^3$ for some $\rho>0$ and we shall use the flow of the functions $H$ to move $\bar f$, it makes sense to consider Hamiltonians $H$ that are all supported inside $B_{2\rho}$.
Hence, one should consider these functions $H$ as the ``tangent space'' at $\bar f$ that will generate the admissible variations. 

Let us introduce the following notation:
$$
\bar f_s^H:=(\Phi_s^H)_\# 
\bar f=\bar f \circ \Phi_{-s}^H\qquad \forall \,s \in \br,
$$
where $s\mapsto \Phi_s^H$ is the Hamiltonian flow of $H$, namely
\begin{equation}
\label{phi}
\left\{ \begin{array}{ccc}\pt_s \Phi^H_s= J \grad H(\Phi^H_s), & J\in \br^{6\times 6}, \quad J=\(\begin{array}{cc}0& Id\\ -Id&0\end{array}\)\\
\Phi_0=Id.
\end{array} \right.
\end{equation}
In other words, $s\mapsto \bar f_s^H$ is the variation generated by $H$, and as $H$ varies, this generates a ``symplectic'' neighborhood of $\bar f$.
Note that, since $\bar f_s^H=\bar f_1^{sH}$, to parameterize a neighborhood of $\bar f$ it is enough to consider the image of the map\footnote{This resembles the exponential map in Riemannian geometry, where a neighborhood of a point $x\in M$ is obtained as the image of a neighborhood of $0$ in $T_xM$ via the map
$$
v\mapsto \gamma_v(1),
$$
where $s\mapsto\gamma_v(s)$ is the geodesic starting at $x$ with velocity $v$.}
\begin{equation}
\label{eq:exp H}
H\mapsto \bar f_1^H.
\end{equation}
We now introduce some definitions to clarify the hypotheses that are needed on the Hamiltonian $H.$
Let us start by defining the set ${\rm Inv}_{\bar f}$ that represents the set of all the Hamiltonians who act trivially on $\bar f$.
\begin{defi}
${\rm Inv}_{\bar f}:=\big\{H\in C^2(\br^6)\,:\, \{H,\bar f\}=0\big\}.$
\end{defi}
This definition is motivated by the following simple result:
\begin{lem}
\label{lem:inv}
If $H \in {\rm Inv}_{\bar f}$ then $(\Phi_t^H)_\#\bar f=\bar f$, i.e., the Hamiltonian flow of $H$ does not move $\bar f$.
\end{lem}
\begin{proof}
The function $\bar f^H_s:=(\Phi_s^H)_\#\bar f$ solves the transport equation
$$
\partial_s\bar f^H_s+\Div(J\nabla H\bar f^H_s)=0,\qquad \bar f^H_s|_{s=0}=\bar f.
$$
Since $J\nabla H\in C^1$ and also $\bar f$ solves this equation (because $\Div(J\nabla H\bar f)=\{H,\bar f\}=0$), we obtain $\bar f_s^H\equiv \bar f$ by uniqueness of solutions to transport equations with $C^1$ vector-fields.
\end{proof}
The lemma above shows that 
if $H$ belongs to $Inv_{\bar f}$, then $\Phi_s^H$ is not moving $\bar f$.
Since our goal is to use Hamiltonians $H$ to parameterize a neighborhood of $\bar f$,  there is no reason to consider $H$ that belong to $Inv(\bar f)$, and it makes sense to exclude them.
Actually, for some technical reasons that will be clearer later, we must impose a quantitative version of the condition $H \not\in Inv_{\bar f}$. To do that, we introduce the family of sets
$$
\mathcal{A}_k:=\{H\in C^2(\br^6) \,:\, \|\grad H \|_{L^1}\le k \|\{H,\bar f\}\|_{L^1} \},\qquad k \geq 1.
$$
\begin{remark}
We note that 
${\rm Inv}_{\bar f}=\underset{k\in \bN}{\bigcap} \mathcal{A}_k^{c}.$
Indeed, if $H \in \underset{k\in \bN}{\bigcap} \mathcal{A}_k^{c}$ then
$\|\grad H \|_{L^1}/k\ge \|\{H,\bar f\}\|_{L^1}$ for all $k \in \bN$.
Thus $\{H,\bar f\}\equiv 0$, which implies that $H \in \rm{Inv}_{\bar f}$.
Viceversa, if $H \in \rm{Inv}_{\bar f}$ then clearly $H \in \underset{k\in \bN}{\bigcap} \mathcal{A}_k^{c}$.
\end{remark}
Because of this observation, we see that 
$$
H \not\in Inv_{\bar f}\qquad \Longleftrightarrow\qquad 
\exists \,k \text{ such that }H\in \mathcal{A}_k.
$$
Motivated by this fact, in the sequel, we shall fix $k$
and consider only Hamiltonians that belong to $\mathcal{A}_k$.
Of course, this is more restrictive than assuming only $H \not\in Inv_{\bar f}$, but at the moment it is not clear to us how to remove such an assumption.
\smallskip

Going further in our preliminary analysis, we observe that all translations of $\bar f$ are trivially stationary solutions. 
However, the kinetic energy automatically controls translations in $v$, and indeed, they do not appear in \eqref{eq:local control}.
To ``kill'' the space of translations in $x$, we will assume that ${\rm Bar}_x(\bar f_1^H)={\rm Bar}_x(\bar f),$ where 
\begin{equation}
\label{eq:bar}
{\rm Bar}_x(f):=\int_{\br^6} x\,f(x,v)\,dx\,dv \in \br^3
\end{equation}
denotes the ``barycenter in $x$'' of $f$.
We want to emphasize that this is not a restrictive assumption on $H$ since one could remove it 
by adding to $H$ a Hamiltonian corresponding to translations in the $x$ variable to recenter the barycenter of $f$. Since this would not add major technical difficulties to the proof but may distract the reader from the essential points, we prefer to impose a barycenter condition on $\bar f_1^H$.
\smallskip

As a final consideration, since our goal is to prove that there are no steady states of  \eqref{gvp} in a neighborhood of $\bar f$ generated via the map \eqref{eq:exp H},
we shall need to assume that our Hamiltonians $H$ are small in some suitable topology.
\smallskip

Our main theorem asserts that, for Hamiltonians small enough in a sufficiently strong Sobolev norm that are also quantitatively away from ${\rm Inv}_{\bar f}$,
there cannot be a stationary point of the form $\bar f_1^H$.

\begin{thm}
\label{thm:main}
Let $\bar f$ be as in \eqref{eq:bar f lemou},
where $F$ is a continuous function from $\br$ to $\br^+$ that satisfies the following monotonicity property: there exists $e_0<0$ such that $F(e)=0$ for $e\ge e_0$ and $F$ is a $C^1$ function on $(-\infty, e_0)$ with $F'<0$ on $(-\infty, e_0).$
Let $\rho>0$ be such that ${\rm supp}(\bar f)\subset B_\rho.$
Also, assume that $\bar f \in W^{2,q}(\br^6)$ for some $q>3$. Then the following local uniqueness result for steady states holds:

Let $r\geq 22$, and
given $\rho,\eps,k>0$ consider the space of functions
$$
\mathcal{N}_\eps^{k}:=\big\{\bar f_1^H
\,:\, {\rm supp}(H)\subset B_{2\rho},\, {\rm Bar}_x(\bar f_1^H)={\rm Bar}_x(\bar f),\, H \in \mathcal{A}_k,\,
\|H\|_{W^{r,2}}
\leq \eps\big\}.
$$
Then, fixed $k \in \bN$, there exists $\eps_k>0$ such that there is no stationary state of \eqref{gvp}
in $\mathcal{N}_{\eps_k}^{k}$ (except for $\bar f$).
\end{thm}

\subsubsection{Comments}
Starting from the seminal paper of Arnold about the geometric interpretation of the Euler equations as $L^2$-geodesics in the space of measure preserving diffeomorphisms \cite{Arn2}, Choffrut and \v{S}ver\'ak obtained
a related result for the 2D Euler equation  \cite{CS}.
The basic idea there is that, under the evolution given by the 2D incompressible Euler equations,
the vorticity is transported by an incompressible vector field, hence the measure of all its super-level sets is constant.
This means that, given an initial vorticity $\omega_0$, its evolution $\omega(t)$ is in the same equimeasurability class of $\omega_0$.
This allows one to foliate the space of vorticities into a family of leaves $\mathcal O_{\omega_0}$ (the equimeasurability class of $\omega_0$), and the Euler equations preserve these leaves.
In addition, thanks to the Hamiltonian structure of the Euler equations, one can characterize stationary solutions as critical points of the Hamiltonian energy $\mathcal E$ restricted to the orbits. 

In other words, one has the following situation: the space of vorticities is foliated by the orbits $\mathcal O_{\omega}$,
and the equilibria are the critical points of $\mathcal E$ restricted to the orbits.
In finite dimension, the implicit function theorem would give the following:
if $\mathcal O_\omega$ is smooth near a point  $\bar \omega\in \mathcal O_\omega$,
and if $\bar\omega$ is a non-degenerate critical point of $\mathcal E$ in $\mathcal O_\omega$,
then near $\bar \omega$, the set of equilibria form a smooth manifold
transversal to the foliation. In addition, the dimension of this manifold is equal to the co-dimension of the orbits.
In particular, in a non-degenerate situation, the equilibria are locally in one-to-one correspondence
with the orbits.
In \cite{CS}, the authors obtain an analog of this correspondence in the infinite-dimensional context of
Euler equations. There, the authors use an infinite dimensional version of the implicit function theorem in the space of $C^\infty$ function via a Nash-Moser's interation.

With respect to their result, here we have different assumptions and results.
These are motivated by the following: 

\begin{itemize}
\item Since the Vlasov-Poisson system \eqref{gvp} is Hamiltonian, given an initial condition $f_0$, its evolution $f_t$ under the Vlasov-Poisson system will also be in the same equimesurability class. However, while Hamiltonian maps and measure preserving maps coincide in 2-dimension, they are very different in higher dimensions (for instance, Hamiltonian maps preserve the symplectic structure). Because solutions to the $3$D Vlasov-Poisson systems describe a Hamiltonian evolution of particles in the phase-space $\mathbb R^3\times\mathbb R^3$, there is no natural reason in this context why there should be only one stationary state in the same equimeasurability class. In particular, as already observed before,
stationary solutions of \eqref{gvp} correspond to critical points of $\mathcal H$ with respect to variations of $\bar f$ generated by Hamiltonian flows, and not with respect to arbitrary measure preserving variations.
This is why we need to look at functions $f$ that can be connected to $\bar f$ via a Hamiltonian flow, namely $f=\bar f_1^H$ for some $H$.
\item The smallness assumption on $\|\nabla^r H\|_{L^2}$ is natural, and actually weaker than the one in \cite{CS}, since smallness there is measured in the $C^\infty$ topology.
We note that we did not try to optimize the value of $r$, which we chose to be equal to $22$ for our proof to work, but that can can probably be optimized further. However, improving this number is not the emphasis here.
In addition, as explained before, the assumption of compact support of $H$ in the definition of $\mathcal{N}_\eps^{k}$ is not restrictive, since $\bar f_1^H$ does not depend on the behaviour of $H$ outside a large ball. 
\item As mentioned before, the assumption on ${\rm Bar}_x({\bar f_1^H})$ is not fundamental: one could remove it by replacing it with $H-H_0$, where $H_0$ corresponds to a translation in the phase space (multiplied by a suitable cut-off function, to make it compactly supported). What is more essential is our assumption $H \in \mathcal{A}_k$, that one would like to replace with $H \not\in Inv_{\bar f}$.
Unfortunately, it is unclear to us how to remove it. 
\end{itemize}

The goal of the next section is to prove our main theorem.

\section{Proof of the Theorem \ref{thm:main}}
\subsection{Strategy of the proof}
The idea of the proof is the following:
first, by exploiting the results in \cite{Lem}, in Lemma~\ref{lem:bar}
we prove that if $\bar f_1^H$
has the same barycenter as $\bar f$, then
$$
\|\bar f^H_1-\bar f\|_{L^1}^2\le C [\mathcal{H}(\bar f^H_1)-\mathcal{H}(\bar f)].
$$
Secondly, in Proposition~\ref{prop:upper} we show that if $\bar f_1^H$ is stationary, then
$$
\mathcal{H}(\bar f^H_1)-\mathcal{H}(\bar f) \leq C\|\nabla H\|_X^3
$$
for some suitable norm $\|\cdot \|_X$ of $\nabla H$.
Combining these two estimates, we get
\begin{equation}
\label{eq:32}
\|\bar f^H_1-\bar f\|_{L^1} \leq C\|\nabla H\|_X^{3/2}.
\end{equation}
As a next step, in Lemma~\ref{lem:L1 g} we prove that
$$
\|\bar f^H_1-\bar f\|_{L^1}=\|\{H,\bar f\}\|_{L^1}+O\big(\|\nabla H\|_X^2\big)=\|\nabla H\cdot J\nabla \bar f\|_{L^1}+O\big(\|\nabla H\|_X^2\big),
$$
that combined with \eqref{eq:32} and our quantitative assumption on the fact that $H$ does not belong to Inv$_{\bar f}$ (namely, $H \in \mathcal{A}_k$) allows us to prove that
$$
\|\nabla H\|_{L^1} \leq C\|\nabla H\|_X^{3/2}.
$$
Finally, exploiting the smallness $\|\nabla^r H\|_{L^2}
\leq \eps$ and interpolation estimates, we relate the two norms above and conclude that
$$
\|\nabla H\|_{L^2} \leq C\|\nabla H\|_{L^2}^{1+\delta}
$$
for some $\delta>0$. This yields a contradiction when $\|\nabla H\|_{L^2}$ is small enough, concluding the proof.

In the following sections, we provide all the details of the argument outlined above.

\subsection{Lower bound}\label{sec:lower}
%

Here and in the sequel, ${\rm Bar}_x$ denotes the spacial barycenter as defined in \eqref{eq:bar}.

\begin{lem}
\label{lem:bar}
Let $\bar f$ be as in \eqref{eq:bar f lemou},
where $F$ is a continuous function from $\br$ to $\br^+$ that satisfies the following monotonicity property: there exists $e_0<0$ such that $F(e)=0$ for $e\ge e_0$ and $F$ is a $C^1$ function on $(-\infty, e_0)$ with $F'<0$ on $(-\infty, e_0).$ 

Let $H \in C^{2}(\br^6)$
with  $\|\nabla H\|_{L^\infty}+\|\nabla^2 H\|_{L^\infty}\leq \eta$, and
consider the function $\bar f^H_1:=(\Phi^H_1)_\#\bar f$. Also, assume that 
\be\label{eq:baricentro}
{\rm Bar}_x(\bar f)={\rm Bar}_x(\bar f^H_1).
\ee
Assume that $\eta$ is small enough so that \eqref{eq:close} holds. Then
\be
\label{eq:lemou}
\|\bar f^H_1-\bar f\|_{L^1}^2\le C [\mathcal{H}(\bar f^H_1)-\mathcal{H}(\bar f)],
\ee
where $C$ depends on the diameter of the support of $\bar f$ and $\bar f_1^H$.
\end{lem}

\begin{proof}
Let $(x_0,0)$ be the point where the minimum is achieved in \eqref{eq:local control}. 
By definition,
\begin{align*}
{\rm Bar}_x(\bar f^H_1(\cdot-(x_0,0)))&=\int_{\br^6}x\bar f^H_1(x-x_0,v)\,dx\,dv\\
&=\int_{\br^6}(x-x_0)\bar f^H_1(x-x_0,v)\,dx\,dv+x_0={\rm Bar}_x(\bar f^H_1)+x_0.
\end{align*}
hence, thanks to \eqref{eq:baricentro},
\begin{align*}
|x_0|&=|{\rm Bar}_x(\bar f^H_1)-{\rm Bar}_x(\bar f^H_1(\cdot-(x_0,0)))|=|{\rm Bar}_x(\bar f)-{\rm Bar}_x(\bar f^H_1(\cdot-(x_0,0)))|\\
&\le\int_{\br^6}|x|\,|\bar f-\bar f^H_1(\cdot-(x_0,v_0))|\,dx\,dz\\
&\le C\|\bar f-\bar f^H_1(\cdot-(x_0,v_0))\|_{L^1}\le C[\mathcal{H}(\bar f^H_1)-\mathcal{H}(\bar f)]^{1/2},
\end{align*}
where we used that $\bar f$ and $\bar f^H_1$ are compactly supported, so $|(x,v)|$ is bounded 
on the support of $\bar f$ and $\bar f^H_1(\cdot-(x_0,0))$.
Thus, using \eqref{eq:local control} we obtain
\begin{align*}
\|\bar f-f^H_1\|_{L^1}&\le \|\bar f-\bar f(\,\cdot+(x_0,0)))\|_{L^1} +  \|\bar f(\,\cdot+(x_0,0)))-\bar f^H_1\|_{L^1}\\
&\le|x_0|\,\|\grad \bar f\|_{L^1} +K_0[\mathcal{H}(\bar f^H_1)-\mathcal{H}(\bar f)]^{1/2}\le C[\mathcal{H}(\bar f^H_1)-\mathcal{H}(\bar f)]^{1/2},
\end{align*}
which concludes the proof.
\end{proof}

%
\subsection{Upper bound}\label{sec:upper}
The aim of this section is to provide an estimate of the difference between the energy of $\bar f$ and of $\bar f_1^H$ in terms $H$,
under the additional assumption that 
$\bar f_1^H$ is a stationary solution for \eqref{gvp}. More precisely, we prove the following:

\begin{prop}
\label{prop:upper}
Let $\bar f$ be a compactly supported steady state such that $\bar f\in L^\infty(\br^6)$,
and that $\bar f\in {W^{2,q}(\br^6)}$ for some $q>3$.
Let $H \in C^2(\br^6)$.
Also, assume that $\bar f_1^H=(\Phi_1^H)_\#\bar f$ is a stationary solution for \eqref{gvp}. Then the following estimate holds:
$$
|\mathcal{H}(f_1^H)-\mathcal{H}(\bar f)|\le C\|\nabla H\|_{L^\infty}\Bigl(\|\nabla H\|_{L^\infty}+\|\nabla^2 H\|_{L^\infty}\Bigr)^2,
$$
where $C$ is a constant depending only on $\bar f$.
\end{prop}

As a first step towards the proof of the above result, we aim to give a characterization of the stationary solutions of \eqref{gvp} in terms of the energy of the system $\mathcal H.$

%

\begin{lem}
\label{lem:stat}
Let $f:\br^6\to \br$ be a compactly supported function.
Then $f$ is a steady state for $\eqref{gvp}$ if and only if
$$
\frac{d}{ds}\mathcal{H}(f_s^H)|_{s=0}=0\qquad \text{for all $H \in C^2(\br^6)$,}
$$
where 
 $f^H_s:=(\Phi_s^H)_\#f$.
\end{lem}

\begin{proof}
Fix $H \in C^2(\br^6)$, and consider its flow $\Phi_s^H$. To simplify the notation we set $\Phi^H_s=\Phi_s$.
Also, it will be convenient to write $\Phi_s=(\Phi_s^x,\Phi_s^v):\br^6\to \br^3\times\br^3$.

Given a compactly supported function
$f$, we compute the first variation of the Hamiltonian $\mathcal H$ around $f$ along $f^H_s:$ recalling that $K(x)=-\frac{1}{4\pi|x|}$ denotes the fundamental solution of the Laplacian, we have
\begin{align*}
\frac{d}{ds}\mathcal{H}(f^H_s) &= \frac{d}{ds}\[\frac{1}{2}\int_{\br^6}|v|^2f_s^H(x,v)\,\,dx\,dv-\frac{1}{2}\int_{\br^{6\times6}}K(x-y)f^H_s(x,v)f^H_s(y,w)\,\,dx\,dv\,dy\,dw\] \\
&=\frac{d}{ds}\bigg[\frac{1}{2}\int_{\br^6}|\Phi^v_s(x,v)|^2f(x,v)\,\,dx\,dv\\
&\qquad-\frac{1}{2}\int_{\br^{6\times6}}K(\Phi^x_s(x,v)-\Phi^x_s(y,v))f(x,v) f(y,w)\,\,dx\,dv\,dy\,dw\bigg] \\
&=\int_{\br^6}\Phi^v_s(x,v)\pt_s\Phi^v_s(x,v)f(x,v)\,\,dx\,dv \\
&\qquad-\frac{1}{2}\biggl[\int_{\br^{6\times6}}\grad_xK(\Phi^x_s(x,v)-\Phi^x_s(y,v))\cdot\\
&\qquad \qquad \qquad\qquad \qquad \qquad \cdot\pt_s (\Phi^x_s(x,v)-\Phi^x_s(y,v)) f(x,v) f(y,w)\,\,dx\,dv\,dy\,dw\biggr] .\\
\end{align*}
Recalling that
$$
\pt_s\Phi_s=\(\grad_v H(\Phi_s),-\grad_x H(\Phi_s)\),
$$
we have 
\begin{align*}
\frac{d}{ds}\mathcal{H}(f^H_s) &=-\int_{\br^6}\Phi^v_s(x,v)\cdot \grad_xH(\Phi_s(x,v)) f(x,v)\,\,dx\,dv \\
&\qquad-\frac{1}{2}\biggl[\int_{\br^{6\times6}}\grad_xK(\Phi^x_s(x,v)-\Phi^x_s(y,v))\cdot\\
&\qquad \qquad \qquad \qquad\cdot(\grad_v H(\Phi_s(x,v))-\grad_vH(\Phi_s(y,v))) f(x,v) f(y,w)\,\,dx\,dv\,dy\,dw\biggr] .\\
\end{align*}
Also, since $K(x-y)=K(y-x),$ we see that $\grad_xK(x-y)=-\grad_xK(y-x),$ so we can rewrite the above expression as
\begin{multline*}
\frac{d}{ds}\mathcal{H}(f^H_s) = -\int_{\br^6}\Phi^v_s(x,v)\cdot\grad_xH(\Phi_s(x,v)) f(x,v)\,\,dx\,dv \\
-\int_{\br^{6\times6}}\grad_xK(\Phi^x_s(x,v)-\Phi^x_s(y,v))\cdot\grad_v H(\Phi_s(x,v)) f(x,v) f(y,w)\,\,dx\,dv\,dy\,dw .
\end{multline*}
Also, using that
$\Phi_s$ preserves the Lebesgue measure
and that $(\Phi_s)^{-1}=\Phi_{-s}$, we can rewrite the first variation in the following way:
\begin{multline}\label{eq:first_var}
\frac{d}{ds}\mathcal{H}(f^H_s) = -\int_{\br^6}v\cdot\grad_xH(x,v) f(\Phi_{-s}(x,v))\,\,dx\,dv \\
-\int_{\br^{6\times6}}\grad_xK(x-y)\cdot\grad_v H(x,v) f(\Phi_{-s}(x,v)) f(\Phi_{-s}(y,w))\,\,dx\,dv\,dy\,dw .
\end{multline}
In particular, since $\Phi_s=Id$ for $s=0$,
we see that
\begin{multline}
\label{eq:first 0}
\frac{d}{ds}\mathcal{H}(f_s)|_{s=0} = -\int_{\br^6}v\cdot\grad_xH(x,v) f(x,v)\,\,dx\,dv \\
-\int_{\br^{6\times6}}\grad_xK(x-y)\cdot\grad_v H(x,y)) f(x,v) f(y,w)\,\,dx\,dv\,dy\,dw.
\end{multline}
On the other hand, for $f$ to be a stationary solution for the system $\eqref{gvp}$ means that
$$
\Div_x(vf)-\Div_v(\grad\phi_{f} f)=0,
$$
or equivalently, that for all $\psi \in C^1$,
\begin{equation}
\label{eq:stationary}
-\int_{\br^6}v\cdot\grad_x \psi(x,v) f(x,v)\,\,dx\,dv+\int_{\br^3}\grad_x\phi_{f}(x)\cdot\grad_v\psi(x,v) f(x,v)\,\,dx\,dv=0.
\end{equation}
Since
$$
\int \grad_xK(x-y)f(y,w)\,dydw=-\nabla \phi_{f}(x),
$$
\eqref{eq:first 0} proves that 
$$
\frac{d}{ds}\mathcal{H}(f_s^H)|_{s=0}=0
\qquad \Longleftrightarrow\qquad \text{\eqref{eq:stationary} holds with $\psi=H$.}
$$
Since $C^2$ functions are dense in $C^1$ for the $C^1$ topology, this proves the result.
\end{proof}

\subsubsection{Second variation for $\mathcal H$}

As a second step, we compute the second variation for $\mathcal H,$ in line with the computation of the first variation $\eqref{eq:first_var}.$ 
Here we consider as initial condition $\bar f$
and, given a Hamiltonian $H \in C^2$, we consider $\bar f_s^H:=\bar f\circ \Phi_{-s}^H$.
As before, to simplify the notation, we set $\Phi_s=\Phi_s^H$.
Also, we define
\begin{equation}\label{eq:g}
g:=\grad \bar f\cdot J \grad H=\{H,\bar f\}
\end{equation}
and we observe that
\begin{equation}\label{eq:d/ds f}
\frac{d}{ds}\bar f(\Phi_{-s})=g(\Phi_{-s}).
\end{equation}
Using \eqref{eq:first_var} and \eqref{eq:d/ds f}, the second variation is given by the following formula:
\begin{align*}
\frac{d^2}{d^2s}\mathcal{H}(\bar f_s^H)&= \[\int_{\br^6}v\cdot\grad_xH(x,v)\,g(\Phi_{-s}(x,v))\,\,dx\,dv\] \\
&+\biggl[\int_{\br^{6\times6}}\grad_{x}K(x-y)\cdot\grad_{v}H(x,v) g(\Phi_{-s}(x,v))\bar f(\Phi_{-s}(y,w))\,\,dx\,dv\,dy\,dw\biggr]\\
&+\biggl[\int_{\br^{6\times6}}\grad_{x}K(x-y)\cdot\grad_{v}H(x,v)\bar f(\Phi_{-s}(x,v))g(\Phi_{-s}(y,w))\,dx\,dv\,dy\,dw\biggr].\\
\end{align*}
Thus, we obtain:
\begin{align}\label{eq:second_var}
&\frac{d^2}{d^2s}\mathcal{H}(\bar f_s^H)=\[\int_{\br^6}v\cdot\grad_xH(x,v)\,g(\Phi_{-s}(x,v))\,\,dx\,dv\]\\ \nn
&+\biggl[\int_{\br^{6\times6}}\grad_{x}K(x-y)\cdot[\grad_{v}H(x,v)-\grad_{w}H(y,w)]\bar f(\Phi_{-s}(x,v))g(\Phi_{-s}(y,w))\,dx\,dv\,dy\,dw\biggr].\\ \nn
\end{align}

\subsubsection{Proof of Proposition \ref{prop:upper}}
As in the previous section, we set 
$\bar f_s^H:=\bar f\circ \Phi_{-s}^H$ and $\Phi_s:=\Phi_s^H$.
Recall that, by assumption, $\bar f_1^H$
is a stationary solution of \eqref{gvp}.

We now study the Taylor expansion of the Hamiltonian of the gravitational Vlasov Poisson system both at $\bar f$ and at $\bar f_1^H$. Since $\bar f$ and $\bar f_1^H$ are two stationary solutions, 
it follows by Lemma \ref{lem:stat} 
applied both to $\bar f$ and to $\bar f_1^H$ that
$$
\frac{d}{ds}\mathcal H(\bar f_s^H)|_{s=0}
=\frac{d}{ds}\mathcal H(\bar f\circ \Phi_{-s})|_{s=0}=0
$$
and
$$
\frac{d}{ds}\mathcal H(\bar f_s^H)|_{s=1}
=\frac{d}{d\tau}\mathcal H(\bar f_1^H\circ \Phi_{-\tau})|_{\tau=0}
=0.
$$
Hence, by Taylor's formula,
 \begin{align*}
\mathcal{H}(\bar f_1^H)=\mathcal{H}(\bar f)+\int_0^1(1-s)\frac{d^2}{d^2s}\mathcal{H}(\bar f_s^H)\,ds
\end{align*}
and
\begin{align*}
\mathcal{H}(\bar f)=\mathcal{H}(\bar f_1^H)+\int_0^1s\frac{d^2}{d^2s}\mathcal{H}(\bar f_s^H)\,ds.
\end{align*}
Therefore, 
\begin{align*}
\mathcal{H}(\bar f_1^H)-\mathcal{H}(\bar f)=\int_0^1(1-2s)\frac{d^2}{d^2s}\mathcal{H}(\bar f_s^H)\,ds.
\end{align*}
Since 
$$
\int_0^1(1-2s)ds=0
$$
we can add a constant term in the integral, and we get
\begin{align}\label{eq:taylor}
\mathcal{H}(\bar f_1^H)-\mathcal{H}(\bar f)=\frac{1}{2}\int_0^1(1-2s)\(\frac{d^2}{d^2s}\mathcal{H}(\bar f_s^H)-\frac{d^2}{d^2s}\mathcal{H}(\bar f_s^H)|_{s=0}\)\,ds.
\end{align}
Thanks to the latter computation, in order to estimate the left hand side of $\eqref{eq:taylor}$, we can estimate
$$
\frac{d^2}{d^2s}\mathcal{H}(\bar f_s^H)-\frac{d^2}{d^2s}\mathcal{H}(\bar f_s^H)|_{s=0}
$$
in terms of the regularity of $H$ and $\bar f$.
Recalling that $\Phi_0=\Id$ and the definition of $g$ in \eqref{eq:g}, we have the following:
\begin{align}\label{eq:stima derivate seconde}
&\frac{d^2}{d^2s}\mathcal{H}(\bar f_s^H) - \frac{d^2}{d^2s}\mathcal{H}(\bar f_s^H)|_{s=0} = \int_{\br^6}v\cdot\grad_xH(x,v)\,g(\Phi_{-s}(x,v))\,\,dx\,dv \\ \nn
& - \int_{\br^6}v\cdot\grad_xH(x,v)\,g(x,v)\,\,dx\,dv\\ \nn
&+\int_{\br^{6\times6}}\grad_{x}K(x-y)\cdot[\grad_{v}H(x,v)-\grad_{w}H(y,w)]\bar f(\Phi_{-s}(x,v))g(\Phi_{-s}(y,w))\,dx\,dv\,dy\,dw\\ \nn
&-\int_{\br^{6\times6}}\grad_{x}K(x-y)\cdot[\grad_{v}H(x,v)-\grad_{w}H(y,w)]\bar f(x,v)g(y,w)\,dx\,dv\,dy\,dw.\\ \nn
\end{align}
Thus,
$$
\bigg|\frac{d^2}{d^2s}\mathcal{H}(\bar f_s^H) - \frac{d^2}{d^2s}\mathcal{H}(\bar f_s^H)|_{s=0}\bigg|\leq T_1+T_2,
$$
where
\begin{align*}
T_1: =\bigg|\int_{\br^6}v\cdot\grad_xH(x,v)\,[g(\Phi_{-s}(x,v)-g(x,v)]\,\,dx\,dv \bigg|,
\end{align*}
\begin{align*}
T_2: &=\biggl|\int_{\br^{6\times6}}\grad_{x}K(x-y)\cdot[\grad_{v}H(x,v)-\grad_{w}H(y,w)]\bar f(\Phi_{-s}(x,v))g(\Phi_{-s}(y,w))\,dx\,dv\,dy\,dw\\
&\qquad-\int_{\br^{6\times6}}\grad_{x}K(x-y)\cdot[\grad_{v}H(x,v)-\grad_{w}H(y,w)]\bar f(x,v)g(y,w)\,dx\,dv\,dy\,dw\biggr|.
\end{align*}
We begin by controlling $T_1$.

By the Fundamental Theorem of Calculus,
\begin{align*}
T_1&\le C\|\grad_xH\|_{L^\infty}\int_{\br^6} \biggl|g(\Phi_{-s}(x,v))-g(x,v)\,\biggr|\,dx\,dv\\
&\le C\|\grad_xH\|_{L^\infty}\int_{\br^6}\biggl(\int_0^1 \biggl|\grad g(\Phi_{-\tau s}(x,v))\cdot \pt_{\tau}\Phi_{-\tau s}(x,v)\,\biggr|\,d\tau\biggr)\,dx\,dv\\
\end{align*}
Using that $\partial_s\Phi_s=J\grad H(\Phi_s)$ and that $\Phi_s$ preserves the volumes, we get
\begin{align*}
T_1&\le C\|\grad H\|^2_\infty\int_{\br^6}\biggl(\int_0^1 \biggl|\grad g(\Phi_{-\tau s}(x,v))\biggr|\,d\tau\biggr)\,dx\,dv\\
&= C\|\grad H\|^2_\infty\int_{\br^6}\biggl(\int_0^1 \biggl|\grad g(x,v)\biggr|\,ds\biggr)\,dx\,dv
\le C\|\grad H\|^2_\infty\|\grad g\|_{L^1}.
\end{align*}
Also, by the definition of $g$ in \eqref{eq:g}, 
\begin{equation}\label{eq:grad g}
\grad g=\nabla^2 H\cdot J\nabla \bar f+\nabla H\cdot J\nabla^2 \bar f,
\end{equation}
therefore
$$
\|\grad g\|_{L^1} \leq \|\nabla^2 H\|_{L^\infty} \|\nabla \bar f\|_{L^1}+\|\nabla H\|_{L^\infty} \|\nabla^2 \bar f\|_{L^1}
\leq C\Bigl(\|\nabla H\|_{L^\infty}+\|\nabla^2 H\|_{L^\infty}\Bigr),
$$
where $C$ depends on $\|\nabla \bar f\|_{L^1}$ and $\|\nabla^2 \bar f\|_{L^1}$.
In conclusion, the first term $T_1$ can be estimate as follows:
\begin{equation}\label{eq:T_1}
T_1 \le  C\|\grad H\|^2_\infty\|\Bigl(\|\nabla H\|_{L^\infty}+\|\nabla^2 H\|_{L^\infty}\Bigr).
\end{equation}

We now estimate the second term:
\begin{align*}
T_2\le \|\grad H\|_{L^\infty}\int_{\br^{6\times6}}|\grad_{x}K(x-y)|\,|\bar f(\Phi_{-s}(x,v))g(\Phi_{-s}(y,w))-\bar f(x,v)g(y,w)|\,dx\,dv\,dy\,dw.
\end{align*}
Adding and subtracting $\bar f(x,v)g(\Phi_{-s}(y,w))$, we can bound
\begin{align*}
T_2&\le\|\grad H\|_{L^\infty}\int_{\br^{6\times6}}|\grad_{x}K(x-y)|\,|\bar f(\Phi_{-s}(x,v))g(\Phi_{-s}(y,w))-\bar f(x,v)g(\Phi_{-s}(y,w))|\,dx\,dv\,dy\,dw\\
&\qquad+\|\grad H\|_{L^\infty}\int_{\br^{6\times6}}|\grad_{x}K(x-y)|\,|\bar f(x,v)g(\Phi_{-s}(y,w))-\bar f(x,v)g(y,w)|\,dx\,dv\,dy\,dw\\
&\le C\|\grad H\|_{L^\infty}\|g\|_{L^\infty}\int_{\br^3\times B_R}|\grad_{x}K(x-y)|\,|\bar f(\Phi_{-s}(x,v))-\bar f(x,v)|\,dx\,dv\,dy\,dw\\
&\qquad+C\|\grad H\|_{L^\infty}\|\bar f\|_{L^\infty}\int_{B_R\times\br^3}|\grad_{x}K(x-y)|\,|g(\Phi_{-s}(y,w))-g(y,w)|\,dx\,dv\,dy\,dw.\\
\end{align*}
Using as before the Fundamental Theorem of Calculus and the fact that $\Phi_s$ is measure preserving, we have
\begin{align*}
T_2&\le C\|\grad H\|_{L^\infty}\|g\|_{L^\infty}\int_{\br^6\times B_R}|\grad_{x}K(x-y)|\cdot \\
&\qquad \qquad \qquad \qquad \qquad \qquad \qquad \cdot\biggl(\int_0^1|\grad \bar f(\Phi_{-\tau s}(x,v))\cdot\pt_{s}\Phi_{-\tau s}(x,v)|\,d\tau\biggr)\,dx\,dv\,dy\,dw\\
&\qquad+C\|\grad H\|_{L^\infty}\|\bar f\|_{L^\infty}\int_{B_R\times\br^6}|\grad_{x}K(x-y)|\cdot \\
&\qquad \qquad \qquad \qquad \qquad \qquad\qquad \cdot\biggl(\int_0^1|\grad g(\Phi_{-\tau s}(y,w))\cdot\pt_{s}\Phi_{-\tau s}(y,w)|\,d\tau\biggr)\,dx\,dv\,dy\,dw.\\
\end{align*}
Using again $\eqref{eq:g}$, we obtain
\begin{align*}
T_2&\le  C\|\grad H\|_{L^\infty}\|g\|_{L^\infty}\int_{\br^6\times B_R}|\grad_{x}K(x-y) |\biggl(\int_0^1|g(\Phi_{-\tau s}(x,v))|\,d\tau\biggr)\,dx\,dv\,dy\,dw\\
&\qquad+C\|\grad H\|_{L^\infty}\|\bar f\|_{L^\infty}\int_{B_R\times\br^6}|\grad_{x}K(x-y)|\cdot \\
&\qquad\qquad \qquad \qquad \qquad \qquad  \cdot\biggl(\int_0^1|\grad g(\Phi_{-\tau s}(y,w))\cdot J\grad H(\Phi_{-\tau s})(-s)|\,d\tau\biggr)\,dx\,dv\,dy\,dw,\\
\end{align*}
that combined with H\"older inequality and the fact that $\Phi_{-\tau s}$ is measure preserving, yields
\begin{align*}
T_2 &\le C\|\grad H\|_{L^\infty}\|g\|^2_\infty\int_{B_R\times B_R}|\grad_{x}K(x-y)|\,dx\,dy\\
&\qquad+C\|\grad H\|^2_\infty\|\bar f\|_{L^\infty}\|\grad g\|_{L^q}\biggl(\int_{B_R\times B_R}|\grad_{x}K(x-y)|^p\,dx\,dy\biggr)^{\frac{1}{p}},
\end{align*}
where $p$ and $q$ are conjugate exponents.
In order to have integrability of the gradient of the kernel $K$, we need $p<\frac{3}{2}.$ Therefore, in the previous estimates we need to assume that $\|\grad g\|_{L^q}$ is finite for some $q>3.$
Thus
\begin{align*}
T_2 &\le C\bigl(\|\grad H\|_{L^\infty}\|g\|^2_{L^\infty} +\|\grad H\|^2_{L^\infty}\|\grad g\|_{L^q}\bigr),\qquad\ q>3,
\end{align*}
where $C$ depends on $\|\bar f\|_{L^\infty}$.
As in the estimate of the term $T_1$ we use $\eqref{eq:g}$ and $\eqref{eq:grad g}$ to get
$$
\|g\|_{L^\infty} \leq \|\nabla H\|_{L^\infty}  \|\nabla \bar f\|_{L^\infty} \leq C\|\nabla H\|_{L^\infty}
$$
and
$$
\|\grad g\|_{L^q} \leq \|\nabla^2 H\|_{L^\infty} \|\nabla \bar f\|_{L^q}+\|\nabla H\|_{L^\infty} \|\nabla^2 \bar f\|_{L^q}
\leq C\Bigl(\|\nabla H\|_{L^\infty}+\|\nabla^2 H\|_{L^\infty}\Bigr),
$$
where $C$ depends on $\|\nabla \bar f\|_{L^q}$ and $\|\nabla^2 \bar f\|_{L^q}$ for some $q>3$.
Since by assumption $\bar f\in {W^{2,q}(\br^6)}$ for some $q>3$,
we have prove that
\begin{align}\label{eq:T_2}
T_2 &\le C\bigl(\|\grad H\|^3_{L^\infty}+\|\grad H\|^2_{L^\infty}\|\grad^2H \|_{L^\infty}\bigr).
\end{align}
Hence, combining $\eqref{eq:stima derivate seconde},$ $\eqref{eq:T_1},$ and $\eqref{eq:T_2}$ we finally obtain
$$
|\mathcal{H}(\bar f_1^H)-\mathcal{H}(\bar f)|\le C\bigl(\|\grad H\|^3_{L^\infty}+\|\grad H\|^2_{L^\infty}\|\grad^2H \|_{L^\infty}\bigr),
$$
where $C$ depends only on $\|\bar f\|_{L^\infty}$, $\|\nabla \bar f\|_{L^q}$, and $\|\nabla^2 \bar f\|_{L^q}$, for some $q>3$. \qed

\subsection{Comparing $\|\bar f-\bar f_1^H\|_{L^1}$ and $\|g\|_{L^1}$}
Our next step is to relate $\|\bar f-\bar f_1^H\|_{L^1}$ and $\|g\|_{L^1}$.

\begin{lem}
\label{lem:L1 g}
Let $\bar f$ be a compactly supported steady state such that $\nabla \bar f,\nabla^2 \bar f\in {L^1(\br^6)}$.
Also, let $H \in C^2(\br^6)$,
and define $g$ as in \eqref{eq:g}.
Set $\bar f_1^H:=\bar f\circ \Phi_{-1}^H$.
Then
$$
\Big|\|\bar f-\bar f_1^H\|_{L^1}-\|g\|_{L^1}\Big|\leq C \|\grad H \|_{L^\infty}\( \|\grad H \|_{L^\infty}+\|\grad^2 H \|_{L^\infty}e^{ \| \grad^2H\|_{L^\infty}}\),
$$
where $C$ depends only on $\bar f$.
\end{lem}
\begin{proof}
Set $\bar f_s^H:=\bar f\circ \Phi_{-s}^H$.
Then, by the definition of the flow $\Phi_s$ (see \eqref{phi}), we deduce
\begin{equation}\label{eq:der f_s}
\pt_s \bar f_s^H=-J\grad H \cdot \grad \bar f_s^H,\qquad \pt_s \bar f_s^H|_{s=0}=-J\grad H \cdot \grad \bar f,
\end{equation}
therefore \eqref{eq:der f_s} and \eqref{eq:g} yield
\begin{align*}
\bar f_1^H-\bar f&=\int_0^1\pt_s \bar f_s^H\,ds=\pt_s \bar f_s^H|_{s=0}+\int_0^1(\pt_s \bar f_s^H-\pt_s \bar f_s^H|_{s=0})\,ds\\
&=-J\grad H\cdot \grad \bar f+\int_0^1(\pt_s \bar f_s^H-\pt_s \bar f_s^H|_{s=0})\,ds\\
&=-g+\int_0^1(\pt_s \bar f_s^H-\pt_s \bar f_s^H|_{s=0})\,ds.
\end{align*}
Thus,
\begin{align*}
\bigg| \int_{\br^6} |\bar f_1^H-\bar f|\,dx\,dv&- \int_{\br^6} |g|\,dx\,dv\bigg|\le\int_0^1\int_{\br^6}|\pt_s \bar f_s^H-\pt_s \bar f_s^H|_{s=0}|\,\,dx\,dv\,ds.
\end{align*}
Using again \eqref{eq:der f_s}, we get
\begin{align*}
&\bigg| \int_{\br^6} |\bar f_1^H-\bar f|\,dx\,dv- \int_{\br^6} |g|\,dx\,dv\bigg|\le\int_0^1\int_{\br^6}|J\grad H \cdot \grad \bar f_s^H-J\grad H \cdot \grad \bar f|\,dx\,dv\,ds\\
&\le \int_0^1ds \int_{\br^6} |\grad H|\, |\grad \bar f\circ \Phi^H_{-s}\cdot\grad \Phi^H_{-s}-\grad \bar f|\,dx\,dv.
\end{align*}
Adding and subtracting $\grad \bar f \circ \Phi^H_{-s}$, this gives
\begin{align*}
\bigg| \int_{\br^6} |\bar f_1^H-\bar f|\,dx\,dv- \int_{\br^6} |g|\,dx\,dv\bigg|&\le \|\grad H \|_{L^\infty} \bigg(\int_0^1 ds \int_{\br^6}|\grad \bar f\circ \Phi^H_{-s}-\grad \bar f |\,dx\,dv\\
&\qquad+ \int_0^1 ds \int_{\br^6} |\grad \bar f\circ \Phi_{-s}^H |\,|\grad \Phi_{-s}^H-\Id| \,dx\,dv\bigg)\\
&=: \|\grad H \|_{L^\infty} (I+II).
\end{align*}
We now estimate the terms $I$ and $II$.
By the Fundamental Theorem of Calculus,
\begin{align*}
I&
=\int_0^1 ds \int_{\br^6}\biggl|\int_0^s\frac{d}{d\tau}\grad \bar f\circ \Phi^H_{-\tau}\,d\tau\biggr|\,dx\,dv\\
&=\int_0^1 ds \int_{\br^6}\biggl|\int_0^s\grad^2 \bar f\circ \Phi^H_{-\tau}\cdot J\nabla H\circ \Phi_{-\tau}^H\,d\tau\biggr|\,dx\,dv\\
&\leq\|\grad H\|_{L^\infty}\int_0^1 ds \int_{\br^6}\int_0^s|\grad^2 \bar f|\circ \Phi^H_{-\tau} \,d\tau\,dx\,dv
\end{align*}
By Fubini, we can rewrite the last integral above as
$$
\int_0^1 ds \int_0^s d\tau \int_{\br^6}|\grad^2 \bar f|\circ \Phi^H_{-\tau} \,dx\,dv
$$
and because $\Phi_{-\tau}^H$ is measure preserving we deduce that the term above is equal to
$$
\int_0^1 ds \int_0^s d\tau \int_{\br^6}|\grad^2 \bar f| \,dx\,dv=\frac{1}{2}\| \grad^2 \bar f\|_{L^1}.
$$
Hence, in conclusion,
$$
I\le \frac12 \| \grad^2 \bar f\|_{L^1} \|\grad H\|_{L^\infty}.
$$
For $II$, we want to estimate the term
$$
|\grad \bar f\circ\Phi_{-s}^H|\,|\grad \Phi_{-s}^H-\Id|.
$$
Differentiating the equation in \eqref{phi}, we deduce that
\begin{equation}\label{eq:grad phi}
 \left\{ \begin{array}{ccc}\partial_s\grad \Phi_s^H=J\grad^2 H(\Phi_s^H)\cdot \grad \Phi_s^H \\
\grad \Phi_0^H(x,v)=\Id.\\
\end{array} \right.
\end{equation}
Thus,
\begin{equation}\label{eq:stima grad phi}
 \left\{ \begin{array}{ccc}\frac{d}{ds}|\grad \Phi_s^H|\le \|\grad^2 H\|_{L^\infty}|\grad \Phi_s^H| \\
|\grad \Phi_0^H|=1\\
\end{array} \right.
\end{equation}
and by Gronwall's inequality 
\begin{equation}\label{eq:gronwall grad phi}
|\grad \Phi_s^H|\le e^{s \| \grad^2H\|_{L^\infty}}.
\end{equation}
Therefore, thanks to \eqref{eq:stima grad phi} and 
\eqref{eq:gronwall grad phi},
\begin{align*}
|\grad \Phi_s^H-\Id|=\biggl|\int_0^s\pt_\tau\grad\Phi_\tau^H\,d\tau\biggr|\leq   \|\grad^2H \|_{L^\infty}\sup_{\tau \in [0,s]}|\grad \Phi_\tau^H|
\leq  \|\grad^2H \|_{L^\infty} e^{s \| \grad^2H\|_{L^\infty}},
\end{align*}
which yields
$$
II\le  \|\grad^2H \|_{L^\infty}e^{ \| \grad^2H\|_{L^\infty}} \int_{\br^6}|\grad \bar f|\,dx\,dv.
$$
Combining the bounds on $I$ and $II$, we conclude that 
$$
\bigg| \int_{\br^6} |f_1-\bar f|\,dx\,dv- \int_{\br^6} |g|\,dx\,dv\bigg|\le C \|\grad H \|_{L^\infty}\( \|\grad H \|_{L^\infty}+\|\grad^2 H \|_{L^\infty}e^{ \| \grad^2H\|_{L^\infty}}\),
$$
where $C$ is a constant depending only on $\|\grad \bar f\|_{L^1}$ and $\|\grad^2 \bar f\|_{L^1}$.
\end{proof}

\subsection{Proof of Theorem \ref{thm:main}}
In this section we combine the upper and lower bounds obtained in Sections \ref{sec:lower} and \ref{sec:upper}  with interpolation estimates to obtain a contradiction to the existence of a stationary solution $\bar f_1^H$,
with $H$ as in the statement of Theorem \ref{thm:main}.

We begin by recalling that, by the Sobolev's embedding, given $R>0$ and $u:B_R\to \br$ compactly supported,
\begin{equation}
\label{eq:Sobolev}
\|u\|_{L^\infty(B_R)} \leq C_{n,R} \|\nabla^s u\|_{L^2(B_R)}\qquad \forall\,s>n/2. 
\end{equation}
In particular, since in our case $n=6$, if $H$ is as in the statement of the theorem then $\|\nabla H\|_{L^\infty}+\|\nabla^2 H\|_{L^\infty}$ is as small a desired provided we choose $\eps$ small enough.
This allows us to apply Lemma \ref{lem:bar}, that combined with Lemma \ref{lem:L1 g} yields following bound on $g$:
$$
\|g\|_{L^1}\le C \(\sqrt{\mathcal{H}(\bar f_1^H)-\mathcal{H}(\bar f)}+ \|\grad H\|^2_{L^\infty}+\|\grad H\|_{L^\infty}\|\grad^2 H \|_{L^\infty}e^{ \| \grad^2H\|_{L^\infty}}\).
$$
Then, using Proposition \ref{prop:upper},
\begin{multline*}
\|g\|_{L^1}\le C \biggl(\|\nabla H\|_{L^\infty}\Bigl(\|\nabla H\|_{L^\infty}+\|\nabla^2 H\|_{L^\infty}\Bigr)^{\frac{1}{2}}\\
+ \|\grad H\|^2_{L^\infty}+\|\grad H\|_{L^\infty}\|\grad^2 H \|_{L^\infty}e^{ \| \grad^2H\|_{L^\infty}}\biggr).
\end{multline*}
We now use the assumption $H \in \mathcal{A}_k$ to get 
\begin{multline}\label{eq:grad H}
\|\nabla H\|_{L^1}\le Ck \bigg(\|\nabla H\|_{L^\infty}\Bigl(\|\nabla H\|_{L^\infty}+\|\nabla^2 H\|_{L^\infty}\Bigr)^{\frac{1}{2}}\\
+ \|\grad H\|^2_{L^\infty}+\|\grad H\|_{L^\infty}\|\grad^2 H \|_{L^\infty}e^{ \| \grad^2H\|_{L^\infty}}\bigg).
\end{multline}
Note that if the norms in the left hand side and in the right hand side were comparable, we would have an inequality of the form
$$
\|\nabla H\|_{X}\le C\|\nabla H\|^{3/2}_{X},
$$
which is impossible when $H$ is small enough.
Thus, the next step is to use interpolation estimates to compare the different norms of $\nabla H$ appearing in \eqref{eq:grad H}.
More precisely, we want to use the following elementary interpolation estimates.
\begin{lem}
\label{lem:interp}
For any smooth compactly supported function $u:\br^n\to \br$,
$$
\|\nabla^{\ell}u \|_{L^2} \leq \|u\|_{L^2}^{1-\ell/m}\|\nabla^{m} u\|_{L^2}^{\ell/m} \qquad \forall\,1\leq \ell\leq m.
$$
\end{lem}
\begin{proof}
The proof is simple using Fourier analysis:
using H\"older inequality with the conjugate exponents $m/\ell$ and $(m-\ell)/\ell$, we have
\begin{align*}
\int |\xi|^{2\ell}|\hat u|^2&=
\int \bigl(|\xi|^{2\ell}|\hat u|^{2\ell/m}\bigr) |\hat u|^{2(m-\ell)/m}
\leq \||\xi|^{2\ell}|\hat u|^{2\ell/m}\|_{L^{m/\ell}} \| |\hat u|^{2(m-\ell)/m}\|_{L^{(m-\ell)/\ell}}\\
&=\biggl(\int |\xi|^{2m}|\hat u|^2\biggr)^{\ell/m}\biggl(\int |\hat u|^2\biggr)^{1-\ell/m}.
\end{align*}
Since $\|\nabla^{k}u \|_{L^2(\br^n)}=\||\xi|^{k}\hat u \|_{L^2(\br^n)}$  for all $k \geq 0$,
the result follows.
\end{proof}
%
Since \eqref{eq:grad H} involves $L^1$ and $L^\infty$ norms, to apply Lemma \ref{lem:interp} we use shall use other interpolation inequalities. More precisely, we recall the classical Nash inequality:
\begin{equation}
\label{eq:Nash}
\|u\|_{L^2(\br^n)}^{1+2/n} \leq C_n\|u\|_{L^1(\br^n)}^{2/n}\|\nabla u\|_{L^2(\br^n)}.
\end{equation} 
We now set $n=6$, and we let $s$ be a number larger than $n/2=3$ to be fixed later.
Applying \eqref{eq:Nash} to $\partial_i H:\br^n\to \br$, $i=1,\ldots,n$, we get
$$
\|\nabla H\|_{L^2}^{1+2/n} \leq C \|\nabla H\|_{L^1}^{2/n}\|\nabla^2 H\|_{L^2}.
$$
Let us recall that, by assumption, $H$ is supported in $B_{2\rho}$.
Hence, we can apply \eqref{eq:Sobolev} both with $u=\partial_iH$ and $u=\partial_{ij}H$ to get
$$
\|\nabla H\|_{L^\infty}\leq C \|\nabla^{s+1} H\|_{L^2},\qquad \|\nabla^2 H\|_{L^\infty}
\leq C\|\nabla^{s+2} H\|_{L^2}.
$$
Note also that, by Poincar\'e inequality in $B_{2\rho}$, $\|\nabla^{s+1} H\|_{L^2}\leq C\|\nabla^{s+2} H\|_{L^2}$.
Combining all these estimates with \eqref{eq:grad H}, we get
\begin{equation}
\label{eq:grad H2}
\|\nabla H\|_{L^2}^{1+2/n}\le Ck \(\|\nabla^{s+2} H\|^{3/2}_{L^2}+\|\nabla^{s+2} H\|^{2}_{L^2}+\|\nabla^{s+2} H\|^{2}_{L^2}e^{ \|\nabla^{s+2} H\|_{L^2}}\)^{2/n}\|\nabla^2 H\|_{L^2}.
\end{equation}
To conclude we recall that, by assumption, $\|H\|_{W^{r,2}} \leq \eps$, where $r\geq 22$,
and we want to obtain a contradiction when $\eps$ is sufficiently small.
To this aim, we first note that, for $s \leq r-2$,
$$
\|\nabla^{s+2} H\|_{L^2} \leq \|H\|_{W^{r,2}} \leq \eps \ll 1.
$$
This implies that the quadratic terms in \eqref{eq:grad H2} are much smaller than the term with the power $3/2$, 
therefore \eqref{eq:grad H2} yields
$$
\|\nabla H\|_{L^2}^{1+2/n}\le Ck \|\nabla^{s+2} H\|_{L^2}^{3/n}\|\nabla^2 H\|_{L^2}.
$$
Then
we apply Lemma \ref{lem:interp} with $u=\partial_iH$, $\ell=s+1$, and $m=r-1$ to get
$$
\|\nabla^{s+2} H\|_{L^2} \leq C\|\nabla H\|_{L^2}^{\frac{r-s-2}{r-1}}\|\nabla^{r} H\|_{L^2}^{\frac{s+1}{r-1}}
\leq C\eps^{\frac{s+1}{r-1}}\|\nabla H\|_{L^2}^{\frac{r-s-2}{r-1}},
$$
therefore
\begin{equation}
\label{eq:gradH3}
\|\nabla H\|_{L^2}^{1+2/n}\le 
Ck \Bigl(\|\nabla H\|_{L^2}^{\frac{r-s-2}{r-1}}\eps^{\frac{s+1}{r-1}} \Bigr)^{3/n}\|\nabla^2 H\|_{L^2}.
\end{equation}
Also, by Lemma \ref{lem:interp} with $u=\partial_iH$, $\ell=1$, $m=r-1$, we have
\begin{equation}
\label{eq:interp H2}
\|\nabla^2 H\|_{L^2} \leq C\|\nabla H\|_{L^2}^{\frac{r-2}{r-1}}\|\nabla^{r} H\|_{L^2}^{\frac{1}{r-1}} \leq C\eps^{\frac{1}{r-1}}\|\nabla H\|_{L^2}^{\frac{r-2}{r-1}}.
\end{equation}
Thus, combining \eqref{eq:gradH3} and \eqref{eq:interp H2}, we obtain
\begin{align*}
&\|\nabla H\|_{L^2}^{1+2/n}\leq Ck \eps^{\frac{1+3(s+1)/n}{r-1}} \|\nabla H\|_{L^2}^{\frac{3(r-s-2)}{n(r-1)}+\frac{r-2}{r-1}}.
\end{align*}
We finally choose $s$. Since $s$ is any exponent larger than $n/2=3$ and less than $r-2 \geq 20$, we fix $s=4$.
Then, the inequality above becomes
\begin{align*}
&\|\nabla H\|_{L^2}^{4/3}\le Ck\eps^{\frac{1+3(s+1)/n}{r-1}} \|\nabla H\|_{L^2}^{\frac{r-6}{2(r-1)} +\frac{r-2}{r-1}}=Ck\eps^{\frac{7}{2(r-1)}} \|\nabla H\|_{L^2}^{\frac{3r-10}{2(r-1)}}.
\end{align*}
Since $r\geq 22$ by assumption, we see that $\frac{3r-10}{2(r-1)}\geq 4/3$, thus we obtain
\begin{align*}
&1\le Ck\eps^{\frac{7}{2(r-1)}} \|\nabla H\|_{L^2}^{\frac{3r-10}{2(r-1)} - \frac{4}3} \leq Ck\eps^{\frac{7}{2(r-1)}+\frac{3r-10}{2(r-1)} - \frac{4}3}=Ck\eps^{1/6},
\end{align*}
which is false for $\eps$ small enough.
This shows the desired contradiction and completes the proof. \qed

------------------



%
%


\bigskip

{\it Acknowledgments:}   The author is grateful to Cl\'ement Mouhot for proposing this problem and for interesting discussions. Also, we wish to thank Pierre R\"aphael for his useful comments during the preparation of this manuscript.
The author would also like to acknowledge the L'Or\'eal Foundation for supporting this project via the L'Or\'eal-UNESCO Award \emph{For Women in Science France fellowship}.

\end{document}